\documentclass[reqno]{amsart}

\usepackage{amsmath}
\usepackage{amsfonts}
\usepackage{amssymb,enumerate}
\usepackage{amsthm,stmaryrd}
\usepackage[all]{xy}
\usepackage{hyperref}
\usepackage{rotating}

\newtheorem{lem}{Lemma} [section]
\newtheorem{cor}[lem]{Corollary}
\newtheorem{prop}[lem]{Proposition}
\newtheorem{thm}[lem]{Theorem}

\newtheorem{Defn}[lem]{Definition}
\newtheorem{Ex}[lem]{Example}
\newtheorem{Question}[lem]{Question}
\newtheorem{Property}[lem]{Property}
\newtheorem{Properties}[lem]{Properties}
\newtheorem{Discussion}[lem]{Remark}
\newtheorem{Construction}[lem]{Construction}
\newtheorem{Notation}[lem]{Notation}
\newtheorem{Fact}[lem]{Fact}
\newtheorem{Notationdefinition}[lem]{Definition/Notation}
\newtheorem{Remarkdefinition}[lem]{Remark/Definition}
\newtheorem{Subprops}{}[lem]
\newtheorem{Para}[lem]{}
\newtheorem{Exer}[lem]{Exercise}
\newtheorem{Exerc}{Exercise}
\newtheorem{Remark}[lem]{Remark}

\newenvironment{defn}{\begin{Defn}\rm}{\end{Defn}}
\newenvironment{ex}{\begin{Ex}\rm}{\end{Ex}}

\newenvironment{property}{\begin{Property}\rm}{\end{Property}}

\newenvironment{notation}{\begin{Notation}\rm}{\end{Notation}}

\newenvironment{notationdefinition}{\begin{Notationdefinition}\rm}{\end{Notationdefinition}}

\newenvironment{disc}{\begin{Discussion}\rm}{\end{Discussion}}

\newenvironment{remark}{\begin{Remark}\rm}{\end{Remark}}

\newtheorem{intthm}{Theorem}

\newcommand{\id}{\operatorname{id}}

\newcommand{\ann}{\operatorname{Ann}}

\newcommand{\ext}{\operatorname{Ext}}	
\newcommand{\rhom}{\mathbf{R}\!\operatorname{Hom}}

\newcommand{\Hom}{\operatorname{Hom}}	
\newcommand{\coker}{\operatorname{Coker}}

\newcommand{\tor}{\operatorname{Tor}}

\newcommand{\Ker}{\operatorname{Ker}}

\newcommand{\ul}{\underline}			

\newcommand{\xra}{\xrightarrow}

\renewcommand{\geq}{\geqslant}
\renewcommand{\leq}{\leqslant}
\renewcommand{\ker}{\Ker}
\renewcommand{\hom}{\Hom}
\renewcommand{\to}{\rightarrow}

\newcommand{\tsr}{\otimes}

\newcommand{\cref}[2]{\ref{#1}(\ref{#1.#2})}

\newcommand{\prn}[1]{\left(#1\right)}

\newcommand{\cppcs}{complete $\mathcal{PP}_C$-resolution }
\newcommand{\cici}{complete $\mathcal{I}_C\mathcal{I}$-resolution}

\newcommand{\cffcs}{complete $\mathcal{FF}_C$-resolution }

\newtheorem*{Para*}{}
\newenvironment{para*}{\begin{Para*}\rm}{\end{Para*}}

\numberwithin{equation}{lem}

\begin{document}

\title{Gorenstein Dimensions over Some Rings of the Form $R \oplus C$}
\author{Pye Phyo Aung}

\address{Department of Mathematics,
NDSU Dept \# 2750,
PO Box 6050,
Fargo, ND 58108-6050
USA}

\email{pye.aung@ndsu.edu}

\urladdr{http://www.ndsu.edu/pubweb/\~{}aung/}

\date{\today}

\keywords{amalgamated duplication, Gorenstein homological dimensions, pseudocanonical cover, retract, semidulaizing modules, trivial extension}

\subjclass[2010]{13D02, 13D05, 13D07}

\maketitle

\begin{abstract}
Given a semidualizing module $C$ over a commutative noetherian ring, Holm and J\o{}rgensen \cite{holm-smarghd} investigate some connections between $C$-Gorenstein dimensions of an $R$-complex and Gorenstein dimensions of the same complex viewed as a complex over the ``trivial extension'' $R \ltimes C$. We generalize some of their results to a certain type of retract diagram. We also investigate some examples of such retract diagrams, namely D'Anna and Fontana's amalgamated duplication \cite{d'anna-adri} and Enescu's pseudocanonical cover \cite{enescu-fclcpc}. 

\end{abstract}

\section{Introduction}
In this paper, $R$ is a commutative noetherian ring with identity.

As in the famous theorem of Auslander-Buchsbaum \cite{auslander-hdilr} and Serre \cite{serre-sldh} where projective dimension of $R$-modules is used to characterize regularity of $R$, Auslander and Bridger introduced Gorenstein dimension in \cite{auslander-smt} to characterize Gorenstein rings: a local ring $R$ is Gorenstein if and only if every finitely generated $R$-module $M$ has finite Gorenstein dimension, i.e., $\operatorname{G-dim}_R M < \infty$ . To extend similar results to non-finitely generated $R$-modules, Enochs and Jenda introduced Gorenstein projective dimension \cite{enochs-gipm}. In particular, a local ring is Gorenstein if and only if every (finitely generated) $R$-module $M$ has finite Gorenstein projective dimension, i.e., $\operatorname{Gpd}_R M < \infty$; see \cite{christensen-gd, enochs-jenda-xu-fdgigp}. Enochs and Jenda also studied the Gorenstein injective dimension $\operatorname{Gid}$ and, with Torrecillas \cite{enochs-jenda-torrecillas-gfm}, the Gorenstein flat dimension $\operatorname{Gfd}$.

Semidualizing $R$-modules, first introduced by Foxby in \cite{foxby-gmrm} and later studied by Vasconcelos \cite{vasconcelos-dtmc} and Golod \cite{golod-gdgpi}, arise naturally in the study of the connection between $R$ and its modules: a finitely generated $R$-module $C$ is \emph{semidualizing} if $R \cong \hom_R(C,C)$ and $\ext_R^i(C,C)=0$ for all $i\geq 1$. For example, Golod introduced the $\operatorname{G}_C$-dimension in \cite{golod-gdgpi} and proved a formula of the same type as the Auslander-Buchsbaum formula and the Auslander-Bridger formula.


Holm and J{\o}rgensen extended the $\operatorname{G}_C$-dimension in \cite{holm-smarghd} by introducing three new homological dimensions called the $C$-Gorenstein projective, $C$-Gorenstein injective and $C$-Gorenstein flat dimensions, denoted as $C\operatorname{-Gpd}_R(M)$, $C\operatorname{-Gid}_R(M)$, and $C\operatorname{-Gfd}_R(M)$, respectively, for an $R$-complex $M$. They also proved how these new dimensions coincide with Enochs, Jenda and Torrecillas' Gorenstein dimensions over the trivial extension $R \ltimes C$ \cite[Theorem 2.16]{holm-smarghd}. This means that for an $R$-module $M$, one has
\begin{align*}
C\operatorname{-Gpd}_R(M) & = \operatorname{Gpd}_{R\ltimes C}(M) \\
C\operatorname{-Gid}_R(M) & = \operatorname{Gid}_{R\ltimes C}(M) \\ 
C\operatorname{-Gfd}_R(M) & = \operatorname{Gfd}_{R\ltimes C}(M).
\end{align*}

In this paper, we generalize the above result to a triple $(R,S,C)$ in the setting of the following retract diagram
$$
\xymatrix{
R \ar[r]^f \ar[rd]_{\id_R} & S \ar[d]^g \\
& R
}
$$
where $R$ and $S$ are commutative rings, $C$ is an $R$-module, $f$ and $g$ are ring homomorphisms, and $\id_R$ is the identity map on $R$, satisfying the following properties:
\begin{enumerate}[(i)]
\item $C\cong \ker g$,
\item $\hom_R(S,C)\cong S$ as $S$-modules, and
\item $\ext_R^i(S,C)=0$ for all $i \geq 1$.
\end{enumerate}
It should be noted here that the triple $(R, R \ltimes C, C)$ satisfies the above conditions if $C$ is semidualizing over $R$. Therefore, the following result generalizes Holm and J{\o}rgensen's Theorem 2.16 \cite{holm-smarghd}; see Theorem~\ref{130904.02}.
\begin{intthm}\label{140404.01}
In the setting of the above retract diagram, if $C$ is a semidualizing $R$-module, then for every homologically left-bounded $R$-complex $M$ and every homologically right-bounded $R$-complex $N$ one has
\begin{align*}
C\operatorname{-Gid}_R(M) & = \operatorname{Gid}_{S}(M) \\ 
C\operatorname{-Gpd}_R(N) & = \operatorname{Gpd}_{S}(N) \\
C\operatorname{-Gfd}_R(N) & = \operatorname{Gfd}_{S}(N).
\end{align*}
\end{intthm}

Along the way we also prove the following characterization of semidualizing modules; see Theorem~\ref{130617.28}.

\begin{intthm}\label{140730.01}
In the setting of the above retract diagram, if $C$ is a finitely generated $R$-module, then the following are equivalent:
\begin{enumerate}[(a)]
\item $C$ is semidualizing over $R$;
\item $R$ is Gorenstein projective over $S$ and $\ann_R(C)=0$; and
\item $C$ is Gorenstein projective over $S$ and $\ann_R(C)=0$.
\end{enumerate}
\end{intthm}

We show that $S=R \ltimes C$ is not the only example of a ring satisfying our generalized settings set forth in the retract diagram above. See Theorems~\ref{130803.06} and \ref{130803.10}, along with their corollaries.

\begin{intthm}\label{140626.03}
The following examples satisfy the hypotheses of Theorem~\ref{140404.01}, i.e. the triple $(R,S,C)$ satisfies the conditions (i) through (iii) in the above retract diagram if we replace $S$ with each of the following rings:
\begin{enumerate}[\rm(a)]
\item \label{140626.03.01} D'Anna and Fontana's \cite{d'anna-adri} amalgamated duplication $S=R \bowtie C$, and
\item \label{140626.03.02} Enescu's \cite{enescu-fclcpc} pseudocanonical cover $S=S(h)$, when $h$ is a square in $R$.
\end{enumerate}
\end{intthm}
In particular, part (a) recovers the main result of Salimi, Tavasoli and Yassemi \cite[Theorem 3.2]{salimi-adrsdi}, and part (b) regarding pseudocanonical covers is a new result.

\section{Preliminaries}
We provide in this section some preliminary definitions and properties to be used later. We first extend a couple of results of Ishikawa \cite{ishikawa-imfm} to our setting.

\begin{lem}\label{130617.31}
Let $f:R \to S$ be a ring homomorphism. Let $S$ be finitely generated as an $R$-module, let $M$ be an $R$-module, and let $N$ be an injective $R$-module. Then the natural map
$$
\Theta_{S,M,N}: S \tsr_R \hom_R(M,N)\to \hom_R(\hom_R(S,M),N)
$$
defined as $(\Theta_{S,M,N}(s\tsr_R \psi))(\phi)=\psi(\phi(s))$ for each $s \tsr_R \psi \in S\tsr_R \hom_R(M,N)$ and each $\phi \in \hom_R(S,M)$, is an $S$-module isomorphism.
\end{lem}

\begin{proof}
By Ishikawa's Hom evaluation result \cite[Lemma 1.6]{ishikawa-imfm}, $\Theta_{S,M,N}$ is an $R$-module isomorphism. One readily checks that $\Theta_{S,M,N}$ is also an $S$-module homomorphism, hence it is an $S$-module isomorphism.
\end{proof}

\begin{lem}\label{140704.01}
Let $f:R \to S$ be a ring homomorphism. Let $S$ be finitely generated as an $R$-module, let $M$ be an $R$-module, and let $N$ be a flat $R$-module. Then the natural map
$$
\Omega_{S,M,N}:\hom_R(S,M) \tsr_R N \to \hom_R(S,M \tsr_R N)
$$
defined as $\Omega_{S,M,N}(\psi \tsr_R n)(s)=\psi(s) \tsr_R n$ for each $\psi \tsr_R n \in \hom_R(S,M) \tsr_R N$ and each $s \in S$, is an $S$-module isomorphism.
\end{lem}
\begin{proof}
The proof is similar to that of Lemma~\ref{130617.31}, using Ishikawa's Tensor evaluation result instead.
\end{proof}

We collect here some properties of injectivity, projectivity and flatness associated with restriction of scalars. A version of this result for $S=R \ltimes C$ is found in \cite[Lemma 3.1]{holm-cmhd}.

\begin{lem} \label{130617.25}
Let $f:R\to S$ be a ring homomorphism.
\begin{enumerate}[\rm(a)]
\item \label{130617.25.03} Each injective $S$-module $J$ is a direct summand of $\hom_R(S,I)$ for some injective $R$-module $I$.
\item \label{130617.25.04} Each projective $S$-module $Q$ is a direct summand of $S \tsr_R P$ for some projective $R$-module $P$.
\end{enumerate}
\end{lem}
\begin{proof}
(a) Since $J$ is also an $R$-module via $f$, we have an exact sequence $0 \to J \to I$ of $R$-modules for some injective $R$-module $I$. Applying the left-exact $\hom_R(S,-)$ to this exact sequence, noting that $\hom_S(S,J)$ is an $S$-submodule of $\hom_R(S,J)$, and using Hom cancellation, we obtain the following $S$-module iso/mono-morphisms.
$$
\xymatrix{
J \xra{\cong} \hom_S(S,J) \ar@{^(->}[r] &  \hom_R(S,J) \ar@{^(->}[r] & \hom_R(S,I).
}
$$
Since $J$ is injective over $S$, this composite monomorphism splits as desired.

(b) This part is proved dually.
\end{proof}

We next define some useful classes of modules and resolutions.

\begin{defn}\label{131212.01}
Let $M$ be an $R$-module, and let $\mathcal{A}$ be a class of $R$-modules. Then an \emph{augmented $\mathcal{A}$-resolution $\ul X ^+$ of $M$} is an exact sequence of $R$-modules of the form
$$
\ul X ^+ = \cdots \xra{\partial_2^X}  X_1 \xra{\partial_1^X} X_0 \to M \to 0
$$
where $X_i \in \mathcal{A}$ for each integer $i \geq 0$. The $R$-complex
$$
\ul X = \cdots \xra{\partial_2^X}  X_1 \xra{\partial_1^X} X_0 \to 0
$$
is the associated \emph{$\mathcal{A}$-resolution} of $M$.
\end{defn}

\begin{defn}\label{131212.02}
Let $N$ be an $R$-module, and let $\mathcal{B}$ be a class of $R$-modules. Then an \emph{augmented $\mathcal{B}$-coresolution $ ^+ \ul Y$ of $N$} is an exact sequence of $R$-modules of the form
$$
^+ \ul Y = \quad 0 \to N \to Y_0 \xra{\partial_0^Y} Y_{-1} \xra{\partial_{-1}^Y} \cdots
$$
where $Y_j \in \mathcal{B}$ for each integer $j \leq 0$. The $R$-complex
$$
\ul Y = \quad 0 \to Y_0 \xra{\partial_0^Y} Y_{-1} \xra{\partial_{-1}^Y} \cdots
$$
is the associated \emph{$\mathcal{B}$-coresolution} of $N$.
\end{defn}

\begin{notation}\label{131212.04}
Let $C$ be an $R$-module.
\begin{enumerate}[\rm(a)]
\item \label{131212.04.02} Let $\mathcal{I}$ be the class of injective $R$-modules.
\item \label{131212.04.01} Let $\mathcal{P}$ be the class of projective $R$-modules.
\item \label{131212.04.03} Let $\mathcal{F}$ be the class of flat $R$-modules.
\item \label{131212.04.05} Let $\mathcal{I}_C$ be the class of $R$-modules isomorphic to $\hom_R(C,I)$ for some injective $R$-module $I$.
\item \label{131212.04.04} Let $\mathcal{P}_C$ be the class of $R$-modules isomorphic to $C \tsr_R P$ for some projective $R$-module $P$.
\item \label{131212.04.06} Let $\mathcal{F}_C$ be the class of $R$-modules isomorphic to $C \tsr_R F$ for some flat $R$-module $F$.
\end{enumerate}
\end{notation}

The following two classes, known collectively as \emph{Foxby classes}, are associated with a finitely generated $R$-module $C$. The definitions can be found in \cite{avramov-rhfgd} and \cite{christensen-sdctac}, and they are studied in conjunction with various homological dimensions, such as the G-dimension in \cite{yassemi-gd}, the $C$-projective dimension in \cite{takahashi-hasm} and the Gorenstein projective dimension in \cite{white-gpdsdm}.

\begin{defn}\label{140626.01}
Let $C$ be a finitely generated $R$-module. The \emph{Auslander class} $\mathcal{A}_C(R)$ is the class of all $R$-modules $M$ such that
\begin{enumerate}[\rm(a)]
\item the natural map $\gamma^C_M:M \to \hom_R(C,C\tsr_R M)$, defined as $\gamma^C_M(m)(c):=c \tsr_R m$ for all $m \in M$ and $c \in C$, is an isomorphism; and
\item $\tor_i^R(C,M)=0=\ext^i_R(C,C\tsr_R M)$ for all $i \geq 1$.
\end{enumerate}
\end{defn}


\begin{defn}\label{140626.02}
Let $C$ be a finitely generated $R$-module. The \emph{Bass class} $\mathcal{B}_C(R)$ is the class of all $R$-modules $M$ such that
\begin{enumerate}[\rm(a)]
\item the evaluation map $\xi^C_M:C \tsr_R \hom_R(C,M) \to M$, defined as $\xi^C_M(c\tsr_R\psi):=\psi(c)$ for all $c \in C$ and $\psi \in \hom_R(C,M)$, is an isomorphism; and
\item $\ext^i_R(C,M)=0=\tor^R_i(C,\hom_R(C,M))$ for all $i\geq 1$.
\end{enumerate}
\end{defn}


\begin{defn}\label{131212.05}
Let $C$ be an $R$-module.

\begin{enumerate}[\rm(a)]
\item \label{131212.05.02} A \emph{complete $\mathcal{I}_C\mathcal{I}$-resolution} $\ul X$ of $R$-modules is an exact sequence of $R$-modules of the form
$$
\ul X = \cdots \xra{\partial_2^X} X_1 \xra{\partial_1^X} X_0 \xra{\partial_0^X} X_{-1} \xra{\partial_{-1}^X} \cdots
$$
such that $X_i \in \mathcal{I}_C$ for each integer $i \geq 1$, $X_j \in \mathcal{I}$ for each integer $j \leq 0$, and $\hom_R(A, \ul X)$ is exact for each $A \in \mathcal{I}_C$.

\item \label{131212.05.01} A \emph{complete $\mathcal{PP}_C$-resolution} $\ul X$ of $R$-modules is an exact sequence of $R$-modules of the form
$$
\ul X = \cdots \xra{\partial_2^X} X_1 \xra{\partial_1^X} X_0 \xra{\partial_0^X} X_{-1} \xra{\partial_{-1}^X} \cdots
$$
such that $X_i \in \mathcal{P}$ for each integer $i \geq 0$, $X_j \in \mathcal{P}_C$ for each integer $j \leq -1$, and $\hom_R(\ul X,A)$ is exact for each $A \in \mathcal{P}_C$.

\item \label{131212.05.03} A \emph{complete $\mathcal{FF}_C$-resolution} $\ul X$ of $R$-modules is an exact sequence of $R$-modules of the form
$$
\ul X = \cdots \xra{\partial_2^X} X_1 \xra{\partial_1^X} X_0 \xra{\partial_0^X} X_{-1} \xra{\partial_{-1}^X} \cdots
$$
such that $X_i \in \mathcal{F}$ for each integer $i \geq 0$, $X_j \in \mathcal{F}_C$ for each integer $j \leq -1$, and $A \tsr_R \ul X$ is exact for each $A \in \mathcal{I}_C$.
\end{enumerate}

\end{defn}

Using complete resolutions, we next define $C$-Gorenstein injectivity, $C$-Gorenstein projectivity and $C$-Gorenstein flatness. We also note that these definitions are equivalent to \cite[Definition 2.7]{holm-smarghd}.

\begin{defn}\label{131213.01}
Let $C$ be an $R$-module. An $R$-module $M$ is
\begin{enumerate}[\rm(a)]
\item \label{131213.01.01} \emph{$C$-Gorenstein injective} if there is a complete $\mathcal{I}_C \mathcal{I}$-resolution $\ul X$, as in Definition~\cref{131212.05}{02}, such that $\ker \partial_0^X$ $\cong M$.
\item \label{131213.01.02} \emph{$C$-Gorenstein projective} if there is a complete $\mathcal{PP}_C$-resolution $\ul X$, as in Definition~\cref{131212.05}{01}, such that $\coker \partial_1^X \cong M$.
\item \label{131213.01.03} \emph{$C$-Gorenstein flat} if there is a complete $\mathcal{FF}_C$-resolution $\ul X$, as in Definition~\cref{131212.05}{03}, such that $\coker \partial_1^X \cong M$.
\end{enumerate}
\end{defn}

When $C=R$, Definition~\ref{131213.01} reduces to the definitions of Gorenstein injectivity, Gorenstein projectivity and Gorenstein flatness of Enochs, Jenda, and Torrecillas \cite{enochs-gipm, enochs-jenda-torrecillas-gfm}, with \cici, \cppcs and \cffcs becoming complete injective resolution, complete projective resolution and complete flat resolution, respectively.

\begin{lem}\label{131226.01}
Let $C$ and $M$ be $R$-modules. Then $M$ is $C$-Gorenstein injective if and only if
\begin{enumerate}[\rm(a)]
\item For each $A \in \mathcal{I}_C$, $\ext_R^i(A,M)=0$ for all $i \geq 1$.
\item $M$ admits an augmented $\mathcal{I}_C$-resolution $\ul Y^+$ such that $\hom_R(A,\ul Y^+)$ is exact for each $A \in \mathcal{I}_C$.
\end{enumerate}
\end{lem}
\begin{lem}\label{131213.03}
Let $C$ and $M$ be $R$-modules. Then $M$ is $C$-Gorenstein projective if and only if
\begin{enumerate}[\rm(a)]
\item For each $A \in \mathcal{P}_C$, $\ext_R^i(M,A)=0$ for all $i \geq 1$.
\item $M$ admits an augmented $\mathcal{P}_C$-coresolution $^+ \ul Y$ such that $\hom_R(^+ \ul Y,A)$ is exact for each $A \in \mathcal{P}_C$.
\end{enumerate}
\end{lem}
\begin{lem}\label{131226.02}
Let $C$ and $M$ be $R$-modules. Then $M$ is $C$-Gorenstein flat if and only if
\begin{enumerate}[\rm(a)]
\item For each $A \in \mathcal{I}_C$, $\tor^R_i(A,M)=0$ for all $i \geq 1$.
\item $M$ admits an augmented $\mathcal{F}_C$-coresolution $^+\ul Y$ such that $ A \tsr_R \prn{^+\ul Y}$  is exact for each $A \in \mathcal{I}_C$.
\end{enumerate}
\end{lem}
Using the fact that $\hom_R(C,I)$ and $I$ are $C$-Gorenstein injective over $R$, one can define a $C$-Gorenstein injective resolution for every homologically left-bounded $R$-complex $M$ \cite[Defintion 2.9]{holm-smarghd}. One can then define the \emph{$C$-Gorenstein injective dimension} of $M$: for a homologically left-bounded $R$-complex $M$, one has
$$
C\operatorname{-Gid}_R M:=\inf\left\{\sup\left\{i \in \mathbb{Z} \mid X_{-i} \neq 0\right\}\left| \begin{matrix}
\text{$X$ is a $C$-Gorenstein injective} \\ \text{resolution of $M$}
\end{matrix} \right.\right\}.
$$
In the case $C=R$, the $C$-Gorenstein injective dimension of $M$ is the \emph{Gorenstein injective dimension} of $M$, denoted $\operatorname{Gid}_R M$. In other words, one has $R\operatorname{-Gid}_R M = \operatorname{Gid}_R M$. The \emph{$C$-Gorenstein projective dimension}, the \emph{Gorenstein projective dimension}, the \emph{$C$-Gorenstein flat dimension}, and the \emph{Gorenstein flat dimension} of an $R$-module $M$, denoted respectively  $C\operatorname{-Gpd}_R M$, $\operatorname{Gpd}_R M$, $C\operatorname{-Gfd}_R M$, and $\operatorname{Gfd}_R M$, are defined similarly.

\section{Semidualizing Modules and Gorenstein Dimensions}

The main point of this section is to prove Theorem~\ref{140404.01} from the introduction.


\begin{property} \label{130619.01}
Let $R$ and $S$ be rings, and let $C$ be an $R$-module. Then the triple $(R,S,C)$ \emph{satisfies Property \ref{130619.01}} if there is a commutative diagram
$$
\xymatrix{
R \ar[r]^f \ar[rd]_{\id_R} & S \ar[d]^g \\
& R
}
$$
of ring homomorphisms with the identity map $\id_R$ on $R$ such that $\hom_R(S,C)\cong S$ as $S$-modules and $\ext_R^i(S,C)=0$ for all $i \geq 1$.
\end{property}

\begin{remark}\label{140619.01}
We first note two immediate facts from the above retract diagram: $R$ is a finitely generated $S$-module and the triple $(R,R \ltimes C,C)$ satisfies Property~\ref{130619.01}. We also note that if a triple $(R,S,C)$ satisfies Property~\ref{130619.01}, then $\rhom_R(S,C)\simeq S$ in the derived category $\mathcal{D}(S)$. In other words, if $\ul I$ is an injective resolution of $C$ over $R$, then Property~\ref{130619.01} implies that $\hom_R(S,\ul I)$ is an injective resolution of the $S$-module $S$. 
\end{remark}

\begin{property} \label{130831.01}
Let $R$ and $S$ be rings, and let $C$ be an $R$-module. Then the triple $(R,S,C)$ \emph{satisfies Property~\ref{130831.01}} if it satisfies Property~\ref{130619.01} and $C \cong \ker g$ as $R$-modules.
\end{property}

\begin{disc}\label{140302.01}
We here note that if $(R,S,C)$ satisfies Property~\ref{130831.01}, it follows that  $S \cong R \oplus C$ as $R$-modules because the short exact sequence $0 \to \ker g \to S \xra{g} R \to 0$ splits.
\end{disc}
We next state and prove versions of several lemmas of Holm and J{\o}rgensen \cite{holm-smarghd, holm-cmhd} in the general setting of Properties \ref{130619.01} and \ref{130831.01}.

\begin{lem}\label{130617.26}
Let $R$ and $S$ be rings, and let $C$ be an $R$-module. If $(R,S,C)$ satisfies Property \ref{130619.01}, then the following statements hold:
\begin{enumerate}[\rm(a)]
\item \label{130617.26.01} For any $R$-module $M$, $\ext_S^i(M,S)\cong \ext_R^i(M,C)$ as $S$-modules for all $i\geq 0$.
\item \label{130617.26.02} For all $i \geq 1$, $\ext_S^i(R,S)=0$ and $\hom_S(R,S)\cong C$ as $S$-modules.
\end{enumerate}
\end{lem}

\begin{proof}
(a) Argue as in \cite[Lemma 3.2 (ii)]{holm-cmhd} with the ring $S$ taking the place of the trivial extension $R \ltimes C$. The essential point is to use Hom-tensor adjointness with the injective resolution $\hom_R(S,\ul I)$ of $S$, as described in Remark~\ref{140619.01}.

(b) This is the special case of part (a) where $M=R$.
\end{proof}

The following is Theorem~\ref{140730.01} from the introduction.

\begin{thm}\label{130617.28}
Let $R$ and $S$ be rings, and let $C$ be a finitely generated $R$-module such that $(R,S,C)$ satisfies Property \ref{130619.01}. Then the following are equivalent:
\begin{enumerate}[(a)]
\item $C$ is semidualizing over $R$;
\item $R$ is Gorenstein projective over $S$ and $\ann_R(C)=0$; and
\item $C$ is Gorenstein projective over $S$ and $\ann_R(C)=0$.
\end{enumerate}
\end{thm}

\begin{proof}
To prove that (a) implies (b), we assume that $C$ is semidualizing over $R$. Using Lemma~\ref{130617.26}, we note that
$$
\ext_S^i(\hom_S(R,S),S)  \cong \ext_S^i(C,S) \cong \ext_R^i(C,C).
$$
This is equal to 0 for all $i \geq 1$ and isomorphic to $R$ when $i=0$ because $C$ is semidualizing over $R$. Again, using the Ext-vanishing from Lemma~\cref{130617.26}{02}, this means that $R$ is Gorenstein projective over $S$ by \cite[Proposition 2.2.2]{christensen-gd}. We also note that $\ann_R(C)$ is the kernel of the homothety map $\chi^R_C:R \to \hom_R(C,C)$, which is 0 because $C$ is semidualizing over $R$.

To prove that (b) implies (c), we recall that $\hom_S(-,S)$ preserves the class of finitely generated Gorenstein projective $S$-modules by \cite[Observation 1.1.7]{christensen-gd}. This proves the desired implication because $C \cong \hom_S(R,S)$ as $S$-modules by Lemma~\cref{130617.26}{02}.

To prove that (c) implies (a), we assume that $C$ is Gorenstein projective over $S$ and $\ann_R(C)=0$. Since $C$ is finitely generated over $R$, it is also finitely generated over $S$. Therefore, by \cite[Theorem 4.2.6]{christensen-gd}, we have
$$
\ext^i_S(C,S)=0=\ext^i_S(\hom_S(C,S),S)
$$
for all $i \geq 1$ and the biduality map
$$
\delta^S_C:C \to \hom_S(\hom_S(C,S),S)
$$
is an $S$-module isomorphism. Using Lemma~\ref{130617.26}, we have
$$
\ext^i_R(C,C)=0=\ext^i_R(\hom_R(C,C),C)
$$
for all $i \geq 1$ and the biduality map 
$$
\delta^S_C:C \to \hom_S(\hom_S(C,S),S)\cong \hom_R(\hom_R(C,C),C)
$$
is an $R$-module isomorphism. Since $\ann_R(C)=0$, it follows that $C$ is semidualizing over $R$ by \cite[Fact 1.1]{ssw-rc}.
\end{proof}

\begin{remark}\label{150114.01}
The assumption $\ann_R(C)=0$ is essential in Theorem~\ref{130617.28}; see \cite[Example 1.2]{ssw-rc}.
\end{remark}

\begin{lem}\label{130617.29}
Let $R$ and $S$ be rings, let $N$ be a finitely generated $R$-module, and let $C$ be a semidualizing $R$-module. If $(R,S,C)$ satisfies Property \ref{130619.01}, and if $N$ is Gorenstein projective as an $S$-module, then the module $\hom_R(N,I)$ is Gorenstein injective over $S$ for any injective $R$-module $I$.
\end{lem}

\begin{proof}
Since $N$ is Gorenstein projective over $S$, the module $N$ has a complete projective resolution $\ul P$ over $S$. Moreover, since $N$ is finitely generated over $R$ (hence over $S$ as well) $\ul P$ can be chosen to consist of finitely generated $S$-modules by \cite[Theorems 4.1.4 and 4.2.6]{christensen-gd}. As in the proof of \cite[Lemma 3.3 (ii)]{holm-cmhd}, it is straightforward to show that $\hom_S(\ul P,\hom_R(S,I))$ is a complete injective resolution of $\hom_R(N,I)$ over $S$.
\end{proof}

We here recover a version of \cite[Lemma 3.3 (ii)]{holm-cmhd} for our general setting.
\begin{prop}\label{130819.01}
Let $R$ and $S$ be rings, and let $C$ be a semidualizing $R$-module. If $(R,S,C)$ satisfies Property~\ref{130619.01}, then for any injective $R$-module $I$, the modules $\hom_R(C,I)$ and $\hom_R(R,I)\cong I$ are Gorenstein injective over $S$.
\end{prop}

\begin{proof}
The modules $C$ and $R$ are Gorenstein projective over $S$ by Theorem~\ref{130617.28}. Thus, the duals $\hom_R(C,I)$ and $\hom_R(R,I)\cong I$ are Gorenstein injective over $S$ by Lemma~\ref{130617.29}.
%
\end{proof}

Next we prove a version of \cite[Lemma 3.4]{holm-cmhd} in the general setting. 

\begin{lem}\label{130617.32}
Let $R$ and $S$ be rings, and let $C$ be a semidualizing $R$-module. If $(R,S,C)$ satisfies Property~\ref{130831.01}, then for any injective $R$-module $J$, we have
$$
\ext^i_{S}(\hom_R(S,J),-)\cong \ext^i_R(\hom_R(C,J),-)
$$
for $i\geq 0$ as functors on $S$-modules.
\end{lem}

\begin{proof}
Argue as in the proof of \cite[Lemma 3.4]{holm-cmhd} that
$$
\hom_R(S,J) \cong \hom_R(\hom_R(S,C),J) \cong S \tsr_R \hom_R(C,J)
$$
as $S$-modules using Lemma~\ref{130617.31} and the fact that $(R,S,C)$ satisfies Property~\ref{130831.01} (hence Property~\ref{130619.01}). If $\ul P$ is a projective resolution over $R$ of $\hom_R(C,J)$, one can argue that $S \tsr_R \ul P$ is a projective resolution over $S$ of $S \tsr_R \hom_R(C,J) \cong \hom_R(S,J)$. This uses the facts that $S \cong R \oplus C$ as $R$-modules and $\hom_R(C,J)$ belongs to $\mathcal{A}_C(R)$ since $J$ is an injective $R$-module. Using this projective resolution over $S$ of $\hom_R(S,J)$ and Hom-tensor adjointness, one can obtain the desired isomorphism.
\end{proof}

As a consequence of the above lemma, we have the following proposition.

\begin{prop}\label{130904.01}
Let $R$ and $S$ be rings, and let $C$ be a semidualizing $R$-module. If $(R,S,C)$ satisfies Property~\ref{130831.01} and $M$ is an $R$-module, then for each $i \geq 0$, we have $\ext_R^i(\hom_R(C,J),M)=0$ for every injective $R$-module $J$ if and only if $\ext_S^i(U,M)=0$ for every injective $S$-module $U$.
\end{prop}

\begin{proof}
As in \cite[Corollary 2.3 (1)]{holm-smarghd}, this follows from Lemmas~\cref{130617.25}{03} and \ref{130617.32}.
%
\end{proof}



\begin{lem}\label{130808.01}
Let $R$ and $S$ be rings, and let $C$ be a semidualizing $R$-module. If the triple $(R,S,C)$ satisfies Property~\ref{130831.01} and $M$ is an $R$-module that is Gorenstein injective over $S$, then there exists a short exact sequence of $R$-modules
$$0 \to M' \to \hom_R(C,I) \to M \to 0$$
for some injective $R$-module $I$ such that
\begin{enumerate}
\item $M'$ is Gorenstein injective over $S$
\item the above sequence is $\hom_R(\hom_R(C,J),-)$-exact for any injective $R$-module $J$.
\end{enumerate}
\end{lem}
\begin{proof}
The proof begins similarly to that of \cite[Lemma 4.1]{holm-cmhd}.

Since $M$ is Gorenstein injective over $S$, it has a complete injective resolution. From this, we can construct the following short exact sequence of $S$-modules
$$
0 \to N \to K \to M \to 0
$$
where $K$ is injective over $S$, $N$ is Gorenstein injective over $S$ and the sequence is $\hom_S(L,-)$-exact for each injective $S$-module $L$, particularly for $L=\hom_R(S,J)$ where $J$ is any injective $R$-module.

As in the proof of \cite[Lemma 4.1]{holm-cmhd}, we can use Lemma~\cref{130617.25}{03} to assume without loss of generality that the above sequence is of the form
\begin{equation}\label{140501.01}
0 \to N \xra{\epsilon} \hom_R(S,I) \xra{\eta} M \to 0
\end{equation}
for some injective $R$-module $I$.

Note that we cannot make use of a specific ring structure of $S$ as in the proof of \cite[Lemma 4.1]{holm-cmhd}, so we use Lemma~\ref{130617.31} instead. Since $S \cong \hom_R(S,C)$ as $S$-modules by Property~\ref{130619.01}, we have
\begin{equation}\label{150114.02}
\hom_R(S,I) \cong \hom_R(\hom_R(S,C),I) \cong S \tsr_R \hom_R(C,I)
\end{equation}
as $S$-modules, where the second isomorphism is by Lemma~\ref{130617.31}. (We note that Lemma~\ref{130617.31} is applicable here because $I$ is injective over $R$, and $S$ is finitely generated over $R$.) The isomorphisms in \eqref{150114.02} enable us to replace $\hom_R(S,I)$ in \eqref{140501.01} with $S\tsr_R \hom_R(C,I)$ to obtain the top row of the following diagram.
\begin{equation}\label{140501.02}
\begin{split}
\xymatrix{
0 \ar[r] & N \ar^-{\epsilon'}[r] \ar^-{\psi \circ \epsilon'}[d] & S \tsr_R \hom_R(C,I) \ar^-{\eta'}[r] \ar^-{\psi}[d] & M \ar[r] \ar@{=}[d] & 0 \\
0 \ar[r] & M':=\ker \phi \ar@{^(->}[r] & \hom_R(C,I) \ar^-{\phi}[r] & M \ar[r] & 0 
}
\end{split}
\end{equation}
The maps $\psi$ and $\phi$ are defined as follows. For any $s \tsr_R \beta \in S \tsr_R \hom_R(C,I)$, set $\psi(s \tsr_R \beta):=s\beta$, where the scalar multiplication is afforded by the $S$-module structure on the $R$-module $\hom_R(C,I)$. For any $\beta$ in $\hom_R(C,I)$, set $\phi(\beta):=\eta'(1_S \tsr_R \beta)$. It is routine to check that both $\psi$ and $\phi$ are well-defined $S$-module homomorphisms and that the diagram \eqref{140501.02} is commutative.

As in \cite[Lemma 4.1]{holm-cmhd}, we can show that the bottom row of the diagram \eqref{140501.02} satisfies the desired properties.
\end{proof}

\begin{lem}\label{130808.02} 
Let $R$ and $S$ be rings, and let $C$ be an $R$-module such that $(R,S,C)$ satisfies Property \ref{130831.01}. Let $M$ be an $R$-module that is $C$-Gorenstein injective over $R$. Then there exists a short exact sequence of $S$-modules 
$$
0 \to M' \to U \to M \to 0
$$
where $U$ is injective over $S$, $M'$ is $C$-Gorenstein injective over $R$ and the above sequence is $\hom_S(V,-)$-exact for any $V$ injective over $S$.
\end{lem}

\begin{proof}
The proof is similar to \cite[Lemma 2.11]{holm-smarghd}, using Lemma~\ref{130617.31} as in the previous result.
\end{proof}

Using the lemmas proved above in the general setting of the retract diagram, we can generalize some propositions and theorems as in \cite{holm-smarghd} and \cite{holm-cmhd}.

\begin{prop}\label{130808.03}
Let $R$ and $S$ be rings, and let $C$ be a semidualizing $R$-module, such that the triple $(R,S,C)$ satisfies Property~\ref{130831.01}. Then, for any $R$-module $M$, $M$ is $C$-Gorenstein injective over $R$ if and only if $M$ is Gorenstein injective over $S$.
\end{prop}

\begin{proof}
This is proved similarly as in \cite[Proposition 2.13 (1)]{holm-smarghd}.
%
%
\end{proof}

We need the dual versions of Lemma~\ref{130617.32}, Proposition~\ref{130904.01}, Lemma~\ref{130808.01} and Lemma~\ref{130808.02} to prove the projective and flat versions of Proposition~\ref{130808.03}. They are stated next for the sake of completeness.

\begin{lem}\label{140302.03}
Let $R$ and $S$ be rings, and let $C$ be a semidualizing $R$-module. If $(R,S,C)$ satisfies Property~\ref{130831.01}, then for any projective $R$-module $Q$, we have
$$
\ext^i_S(-,S\tsr_R Q) \cong \ext^i_R(-,C\tsr_R Q)
$$
for all $i \geq 0$ as functors on $S$-modules.
\end{lem}
\begin{proof}
This is the dual of Lemma~\ref{130617.32} using Lemma~\ref{140704.01} and $\hom_R(S,\ul I)$ as the injective resolution over $S$ of $\hom_R(S,C \tsr_R Q)$ where $\ul I$ is an injective resolution of $C \tsr_R Q$.
\end{proof}

\begin{prop}\label{140302.04}
Let $R$ and $S$ be rings, and let $C$ be a semidualizing $R$-module. If $(R,S,C)$ satisfies Property~\ref{130831.01} and $M$ is an $R$-module, then for each $i \geq 0$, we have $\ext^i_R(M,C \tsr_R P)=0$ for every projective $R$-module $P$ if and only if $\ext^i_S(M,V)=0$ for every projective $S$-module $V$.
\end{prop}

\begin{proof}
This is the dual of Proposition~\ref{130904.01}.
\end{proof}

\begin{lem}\label{140302.05}
Let $R$ and $S$ be rings, and $C$ be a semidualizing $R$-module. If the triple $(R,S,C)$ satisfies Property~\ref{130831.01} and $M$ is an $R$-module that is Gorenstein projective over $S$, then there exists a short exact sequence of $R$-modules
$$
0 \to M \to C \tsr_R P \to M' \to 0
$$
for some projective $R$-module $P$ such that
\begin{enumerate}
\item $M'$ is Gorenstein projective over $S$
\item the above sequence is $\hom_R(-,C\tsr_R Q)$-exact for any projective $R$-module $Q$.
\end{enumerate}
\end{lem}
\begin{proof}
This is the dual of Lemma~\ref{130808.01}, using Lemma~\ref{140704.01}.
\end{proof}

\begin{lem}\label{140302.06}
Let $R$ and $S$ be rings, and let $C$ be an $R$-module such that $(R,S,C)$ satisfies Property~\ref{130831.01}. Let $M$ be an $R$-module that is $C$-Gorenstein projective over $R$. Then there exists a short exact sequence of $S$-modules
$$
0 \to M \to W \to M' \to 0
$$
where $W$ is projective over $S$, $M'$ is $C$-Gorenstein projective over $R$ and the above sequence is $\hom_S(-,Y)$-exact for any projective $S$-module $Y$.
\end{lem}

\begin{proof}
This is the dual of Lemma~\ref{130808.02}.
\end{proof}

Using the above results, one can prove the injective version of Proposition~\ref{130808.03}.

\begin{prop}\label{140302.07}
Let $R$ and $S$ be rings, and let $C$ be a semidualizing $R$-module, such that the triple $(R,S,C)$ satisfies Property~\ref{130831.01}. Then, an $R$-module $M$ is $C$-Gorenstein projective if and only if $M$ is Gorenstein projective over $S$.
\end{prop}
\begin{proof}
Argue similarly as in the proof of Proposition~\ref{130808.03} using Lemmas~\ref{140302.03}, \ref{140302.05}, \ref{140302.06} and Proposition~\ref{140302.04} instead.
\end{proof}

The flat version of Proposition~\ref{130808.03} can be proved by essentially the same techniques as in the proof of \cite[Proposition 2.15]{holm-smarghd}.


\begin{prop}\label{140302.08}
Let $R$ and $S$ be rings, and let $C$ be a semidualizing $R$-module, such that the triple $(R,S,C)$ satisfies Property~\ref{130831.01}. Then, for any $R$-module $M$, $M$ is $C$-Gorenstein flat over $R$ if and only if $M$ is Gorenstein flat over $S$.
\end{prop}
\begin{proof}
Argue as in the beginning of the proof of \cite[Proposition 2.15]{holm-smarghd}, using Hom-tensor adjointness, that for any faithfully injective $R$-module $E$, the module $M$ is $C$-Gorenstein flat if and only if the module $\hom_R(M,E)$ is $C$-Gorenstein injective.

Since $\hom_R(S,E)$ is faithfully injective over $S$ for any faithfully injective $R$-module $E$, one has $\operatorname{Gfd}_S M=\operatorname{Gid}_S(\hom_S(M,\hom_R(S,E)))$ by \cite[Theorem 6.4.2]{christensen-gd}. Moreover, since $\hom_S(M, \hom_R(S,E))\cong \hom_R(M, E)$ by Hom-tensor adjointness and tensor cancellation, we have $\operatorname{Gfd}_S M=\operatorname{Gid}_S(\hom_R(M,E))$.

The above two facts, combined with Proposition~\ref{130808.03}, give the desired result.
\end{proof}

The last result of this section is Theorem~\ref{140404.01}.

\begin{thm}\label{130904.02}
Let $R$ and $S$ be rings, and let $C$ be a semidualizing $R$-module. If $(R,S,C)$ satisfies Property~\ref{130831.01}, then for any homologically left-bounded $R$-complex $M$ and any homologically right-bounded $R$-complex $N$, one has
\begin{align*}
C\operatorname{-Gid}_R M & = \operatorname{Gid}_S M \\
C\operatorname{-Gpd}_R N & = \operatorname{Gpd}_S N \\
C\operatorname{-Gfd}_R N & = \operatorname{Gfd}_S N 
\end{align*}
\end{thm}

\begin{proof}
This follows from Propositions~\ref{130808.03}, \ref{140302.07} and \ref{140302.08} as in \cite[Theorem 2.16]{holm-smarghd}. We prove the equality $C\operatorname{-Gid}_R M = \operatorname{Gid}_S M$ as follows.

Since every $C$-Gorenstein injective $R$-module is Gorenstein injective over $S$ by Proposition~\ref{130808.03}, we have $C\operatorname{-Gid}_R M \geq \operatorname{Gid}_S M$.

For the opposite inequality, let $n:=\operatorname{Gid}_S M < \infty$. Let $I$ be a left-bounded complex of injective $R$-modules such that $I \simeq M$ in the derived category $\mathcal{D}(R)$. Each $I_i$ is injective over $R$, hence Gorenstein injective over $S$ by Proposition~\ref{130819.01}. Let $M':=\ker(I_{-n} \to I_{-n-1})$. By \cite[Theorem 3.3]{christensen-gpifd}, $M'$ is Gorenstein injective over $S$. By Proposition~\ref{130808.03}, $M'$ is $C$-Gorenstein injective over $R$, and so are $I_0, \ldots, I_{-n+1}$. Let $I':= \cdots \to I_{-n+1} \to M' \to 0$. Since $I'\simeq I \simeq M$, $C\operatorname{-Gid}_R M \leq n$.

The projective and flat versions are proved similarly.
\end{proof}

\section{Examples}

It is routine to show that Nagata's trivial extension $R \ltimes C$ satisfies Property~\ref{130831.01}, hence we can recover \cite[Theorem 2.16]{holm-smarghd} as a special case of Theorem~\ref{130904.02}. The rest of this section is devoted to two similar constructions, which can be recovered as special cases of our retract diagram in Property~\ref{130619.01}. In particular, we prove in this section Theorem~\ref{140626.03} from the introduction.

\subsection{Amalgamated Duplication of a Ring along an Ideal}

\

The following construction is due to D'Anna and Fontana \cite{d'anna-adri}.

\begin{Notationdefinition}\label{130803.02}
Let $R$ be a ring, and let $C$ be an ideal in $R$. Then define a multiplication structure on $R \oplus C$ as follows: for each $(r,c)$ and $(r',c')$ in $R \oplus C$, we define $(r,c)(r',c')=(rr',rc'+r'c+cc')$. The group $R \oplus C$ with this multiplication structure is a ring with $(1_R,0)$ as the multiplicative identity~\cite{d'anna-adri}. We denote this ring as $R \bowtie C$, and call it the amalgamated duplication of $R$ along $C$.
\end{Notationdefinition}

It is routine to check that $R \bowtie C$ fits into a  retract diagram satisfying Property~\ref{130619.01}. We collect this information in the following lemma.

\begin{lem}\label{130803.03}
Let $R$ be a ring, and let $C$ be an ideal in $R$. Then the diagram
$$
\xymatrix{
R \ar^-{f}[r] \ar_{\id_R}[rd] & R\bowtie C \ar^{g}[d] \\
& R
}
$$
where $f(r):=(r,0)$ and $g(r,c):=r$ for each $r \in R$ and $c \in C$, is a commutative diagram of ring homomorphisms such that $\ker g \cong C$ as $R$-modules.
\end{lem}
%
%
%

We prove next that the triple $(R, R \bowtie C, C)$ satisfies Property~\ref{130619.01}.

%
%
%

\begin{lem}\label{130803.05}
Let $R$ be a ring, and let $C$ be an ideal in $R$. If $C$ is a semidualizing $R$-module, then $\hom_R(R \bowtie C,C) \cong R \bowtie C$ as $R\bowtie C$-modules, and \mbox{$\ext_R^i(R \bowtie C,C)=0$} for all $i \geq 1$.
\end{lem}

\begin{proof} We first note that the $R \bowtie C$-module structure of $\hom_R(R \bowtie C,C)$ comes from $R \bowtie C$ in the first slot. Specifically, for any $(r,c)$ and $(s,d)$ in $R\bowtie C$, and for any $R$-module homomorphism $\varphi$ from $R\bowtie C$ to $C$, we have $((r,c)\varphi)(s,d)=\varphi((r,c)(s,d))=\varphi(rs,rd+sc+cd)$. Since $R \bowtie C \cong R \oplus C$ as $R$-modules, we know that $\hom_R(R \bowtie C,C) \cong \hom_R(C,C) \oplus C$ as $R$-modules. 

Since $C$ is assumed to be semidualizing over $R$, the natural homothety $R \to \hom_R(C,C)$ is an $R$-module isomorphism, hence there is a natural $R$-module isomorphism $\Theta: R \bowtie C \to \hom_R(R \bowtie C,C)$. Tracing all the natural isomorphisms involved, we see that the natural $R$-module isomorphism $\Theta:R \bowtie C \to \hom_R(R \bowtie C,C)$ is defined by $(r,c) \mapsto \phi^{(r,c)}$, where $\phi^{(r,c)}$ is defined for any $(r'',c'') \in R \bowtie C$ as $\phi^{(r,c)}(r'',c'')=rc''+r''c$.

However, unlike in the case of $R \ltimes C$, the natural $R$-module isomorphism $\Theta$ is \textit{not} an $R \bowtie C$-module isomorphism. We therefore use $\Theta$ to construct a new map $\Phi$ from $R \bowtie C$ to $\hom_R(R\bowtie C, C)$, and we prove that $\Phi$ is indeed an $R\bowtie C$-module isomorphism.

Define $\Phi:R \bowtie C \to \mbox{$\hom_R(R\bowtie C,C)$}$ as $\Phi(r,c):=\Theta(r+c,c)$ for any $(r,c)$ in $R\bowtie C$. It is then immediate that $\Phi$ is bijective, and it is also routine to check that $\Phi$ is indeed an $R \bowtie C$-module homomorphism with respect to the module structures noted above. 

Finally, we note that we already have $\ext_R^i(R \bowtie C,C)\cong \ext_R^i(C,C)$ as $R$-modules. Since $C$ is semidualizing over $R$, we have $\ext_R^i(C,C)\cong 0$ for all $i \geq 1$, and $\ext_R^i(R \bowtie C, C)\cong 0$ as well.
\end{proof}

The next result is Theorem~\cref{140626.03}{01} from the introduction.

\begin{thm}\label{130803.06}
Let $R$ be a ring, let $C$ be an ideal in $R$, and set $S:=R\bowtie C$. If $C$ is semidualizing as an $R$-module, then $(R,S,C)$ satisfies Property~\ref{130831.01}.
\end{thm}

\begin{proof}
Lemma~\ref{130803.03} combined with Lemma~\ref{130803.05} yields the desired result.
\end{proof}

Since $(R,R\bowtie C,C)$ satisfies Property~\ref{130831.01}, Theorem~\ref{130904.02} can be applied to reover the following result of Salimi et. al. \cite{salimi-adrsdi}.

\begin{cor}\label{140501.03}
Let $R$ be a ring, and let $C$ be an ideal in $R$ such that $C$ is semidualizing over $R$. Then, for any homologically left-bounded $R$-complex $M$ and any homologically right-bounded $R$-complex $N$, one has
\begin{align*}
C\operatorname{-Gid}_R M & = \operatorname{Gid}_{R \bowtie C} M \\
C\operatorname{-Gpd}_R N & = \operatorname{Gpd}_{R \bowtie C} N \\
C\operatorname{-Gfd}_R N & = \operatorname{Gfd}_{R \bowtie C} N 
\end{align*}
\end{cor}

\subsection{Pseudocanonical Cover}

\

In this section, we apply Theorem~\ref{130904.02} to pseudocanonical covers introduced by Enescu in \cite{enescu-fclcpc}.

\begin{notationdefinition}\label{130803.07}
Let $R$ be a ring, let $h \in R$, and let $C$ be an ideal in $R$. We define a ring structure on $R \oplus C$ by defining $(r,c)(r',c')=(rr'+cc'h,rc'+r'c)$ for each $(r,c),(r',c') \in R \oplus C$. The group $R \oplus C$ with this multiplication structure, denoted as $S(h)$, is indeed a ring with $(1_R,0)$ as its multiplicative identity~\cite{enescu-fclcpc}, and is called the \emph{pseudocanonical cover of $R$ via $h$}.
\end{notationdefinition}

We construct a retract diagram similar to the one in Property~\ref{130619.01} using $S(h)$.

\begin{lem}\label{130803.08}
Let $R$ be a ring, let $C$ be an ideal and let $h \in R$ such that $h=r_0^2$ for some $r_0 \in R$. Then the diagram
$$
\xymatrix{
R \ar^{f}[r] \ar_{\id_R}[rd] & S(h) \ar^{g}[d] \\
& R
}
$$
where $f(r):=(r,0)$ and $g(r,c):=r+cr_0$ for each $r \in R$ and $c \in C$, is a commutative diagram of ring homomorphisms such that $\ker g \cong C$ as $R$-modules.
\end{lem}

\begin{proof}
By construction, $f$ and $g$ are well-defined functions making the diagram commute. It is routine to check that $f$ is a ring homomorphism and that $g$ respects addition. To check that $g$ respects multiplication as well, let $r,r' \in R$ and $c,c'\in C$. Then
\begin{align*}
g\left((r,c)(r',c')\right) & = g(rr'+cc'h,rc'+r'c) \\
& = rr'+cc'h+rc'r_0+r'cr_0 \\
& =r(r'+c'r_0)+cc'r_0^2+r'cr_0 \\
& = r(r'+c'r_0)+cr_0(c'r_0+r') \\
& = (r+cr_0)(r'+c'r_0) \\
& = g(r,c)g(r',c')
\end{align*}
where we used the fact that $h=r_0^2$.

We note that $\ker g$ is the $R$-submodule of $S(h)$ consisting of all elements of the form $(-cr_0,c)$ with $c \in C$. Therefore one can readily prove that the map from $C$ to $\ker g$ that sends $c$ to $(-cr_0,c)$ is indeed an $R$-module isomorphism.
\end{proof}

\begin{lem}\label{130803.09}
Let $R$ be a ring, let $C$ be an ideal in $R$, and let $h \in R$ such that $h=r_0^2$ for some $r_0 \in R$. If $C$ is semidualizing over $R$, then $\hom_R(S(h),C) \cong S(h)$ as $S(h)$-modules, and $\ext_R^i(S(h),C)=0$ for all $i \geq 1$.
\end{lem}

\begin{proof} We first note that the $S(h)$-module structure of $\hom_R(S(h),C)$ comes from $S(h)$ in the first slot. Since $S(h) \cong R \oplus C$ as $R$-modules, we know that
$$
\hom_R(S(h),C) \cong \hom_R(C,C) \oplus C
$$
as $R$-modules.

Since $C$ is assumed to be semidualizing over $R$, we have $\hom_R(C,C) \cong R$ as $R$-modules, hence $S(h)\cong \hom_R(S(h),C)$ as $R$-modules. Tracing the composition of all the natural $R$-module isomorphisms above, we have an $R$-module isomorphism $\Theta:S(h)\to \hom_R(S(h),C)$ sending $(r,c) \mapsto \varphi^{(r,c)}$, where $\varphi^{(r,c)}$ is defined for any $(r'',c'') \in S(h)$ as $\varphi^{(r,c)}(r'',c'')=rc''+r''c$. It is routine to check that $\Theta$ is also an $S(h)$-module homomorphism.

%
%

Finally, we have that $\ext_R^i(S(h),C)\cong \ext_R^i(C,C)$ as $R$-modules for all $i \geq 1$. Since $C$ is semidualizing over $R$, we have $\ext_R^i(C,C) = 0$ for all $i \geq 1$, hence $\ext_R^i(S(h), C) = 0$ as well.
\end{proof}

The next result is Theorem~\cref{140626.03}{02} from the introduction.

\begin{thm}\label{130803.10}
Let $R$ be a ring, let $C$ be an ideal in $R$, let $h \in R$ such that $h=r_0^2$ for some $r_0 \in R$, and let $S(h)$ be the pseudocanonical cover of $R$ via $h$. If $C$ is semidualizing as an $R$-module, then $(R,S(h),C)$ satisfies Property~\ref{130831.01}.
\end{thm}

\begin{proof}
Lemma~\ref{130803.08} combined with Lemma~\ref{130803.09} yields the desired result.
\end{proof}

We can apply Theorem~\ref{130904.02} to $S(h)$.

\begin{cor}\label{140501.04}
Let $R$ be a ring, let $C$ be an ideal in $R$ such that $C$ is semidualizing over $R$, and let $h \in R$ such that $h=r_0^2$ for some $r_0\in R$. Then, for any homologically left-bounded $R$-complex $M$ and any homologically right-bounded $R$-complex $N$, one has
\begin{align*}
C\operatorname{-Gid}_R M & = \operatorname{Gid}_{S(h)} M \\
C\operatorname{-Gpd}_R N & = \operatorname{Gpd}_{S(h)} N \\
C\operatorname{-Gfd}_R N & = \operatorname{Gfd}_{S(h)} N 
\end{align*}
\end{cor}
\begin{proof}
Since $(R,S(h),C)$ satisfies Property~\ref{130831.01}, this is a direct application of Theorem~\ref{130904.02}.
\end{proof}

\section{Final Remarks}

It is natural to ask if the general settings we discussed characterizes the situation where an $R$-module $M$ is $C$-Gorenstein projective/injective/flat over $R$ if and only if $M$ is Gorenstein projective/injective/flat over $S$. For example, one can ask if Property~\ref{130831.01} is a necessary condition for Theorem~\ref{140302.07}. This fails in general, and the following is a counterexample.

\begin{ex}\label{140704.02}
Let $C$ be a semidualizing module, and set $R_1:=R \ltimes C$ and $S:=R_1 \ltimes R_1$. Let $g_1$ and $g_2$ be appropriate canonical projections, making the following diagram commute, and let $g:=g_2 \circ g_1$.
$$
\xymatrix{
R \ar[rdd] \ar[r]& S \ar^-{g_1}[d]\\
 & R_1 \ar^-{g_2}[d] \\
 &R
}
$$
We note that $M$ is $C$-Gorenstein projective over $R$ if and only if it is Gorenstein projective over $R_1$, if and only if Gorenstein projective over $S$ by \cite[Proposition 2.13]{holm-smarghd}. However, the triple $(R,S,C)$ does not satisfy Property~\ref{130831.01} because $\ker g \not \cong C$; noting that $S\cong R \oplus (C \oplus R \oplus C)$ as $R$-modules and $g=g_2 \circ g_1$, we have $\ker g \cong R \oplus C^2$, which is different from $C$. 
\end{ex}

We finally note here that the $R$-module structure on $S$ in the previous example is not by accident. If we assume that a retract diagram in our general setting exists, i.e., there exists a ring homomorphism $f: R \to S$ such that $g \circ f=\id_R$, then $g$ is a split surjection. This implies that $S \cong R \oplus \ker g$ as $R$-modules as in the above example.

\section*{Acknowledgements}
We are grateful to the anonymous referee for his or her detailed comments.
\bibliographystyle{amsplain}

\providecommand{\bysame}{\leavevmode\hbox to3em{\hrulefill}\thinspace}
\providecommand{\MR}{\relax\ifhmode\unskip\space\fi MR }
\providecommand{\MRhref}[2]{%
  \href{http://www.ams.org/mathscinet-getitem?mr=#1}{#2}
}
\providecommand{\href}[2]{#2}

\end{document}